\documentclass[12pt,twoside,reqno]{amsart}

\usepackage{version}         

\usepackage{color}

\usepackage{multirow}
\usepackage[OT1]{fontenc}
\usepackage{type1cm}
\usepackage{amssymb}
\usepackage{mathrsfs}
\usepackage[english]{babel}
\usepackage[latin1]{inputenc}
\usepackage[T1]{fontenc}
\usepackage{palatino}
\usepackage{amsfonts}
\usepackage{amsmath}
\usepackage{amssymb}
\usepackage{amsthm}
\usepackage{bbm}
\usepackage{extarrows}
\usepackage{graphicx}
\usepackage[usenames,dvipsnames]{xcolor}
\usepackage{wrapfig}
\usepackage{type1cm}
\usepackage[bf]{caption}
\usepackage{esint}
\usepackage{version}
\usepackage[colorlinks = true, citecolor = black, linkcolor = black, urlcolor = black, pdfstartview=FitH]{hyperref}
\usepackage{caption}
\usepackage{tikz}
\usepackage{bbding, wasysym}
\usepackage{enumitem}

\usepackage[left=2.75cm,top=3cm,right=2.75cm]{geometry}
\allowdisplaybreaks[4]

\numberwithin{equation}{section}

\theoremstyle{plain}
\newtheorem{theorem}{Theorem}[section]

\newtheorem{proposition}[theorem]{Proposition}
\newtheorem{lemma}[theorem]{Lemma}

\theoremstyle{remark}
\newtheorem{remark}[theorem]{Remark}

\newtheorem*{ack}{Acknowledgement}

\theoremstyle{definition}
\newtheorem{definition}[theorem]{Definition}

\newcommand{\R}{\mathbb{R}}

\newcommand{\N}{\mathbb{N}}

\newcommand{\w}{\mathbf{w}}

\newcommand{\cE}{\mathcal{E}}
\newcommand{\cM}{\mathcal{M}}
\newcommand{\cH}{\mathcal{H}}

\newcommand{\cF}{\mathcal{F}}

\newcommand{\cP}{\mathcal{P}}

\newcommand{\cA}{\mathcal{A}}

\renewcommand{\emptyset}{\varnothing}
\renewcommand{\epsilon}{\varepsilon}
\renewcommand{\rho}{\varrho}
\renewcommand{\phi}{\varphi}

\renewcommand{\i}{\mathtt{i}}

\begin{document}

\title{Hausdorff dimension of frequency sets in beta-expansions}

\author{Yao-Qiang Li}

\address{School of Mathematics \\
         South China University of Technology \\
         Guangzhou, 510641 \\
         P.R. China}
\email{scutyaoqiangli@qq.com}

\address{Institut de Math\'ematiques de Jussieu - Paris Rive Gauche \\
         Sorbonne Universit\'e - Campus Pierre et Marie Curie\\
         Paris, 75005 \\
         France}
\email{yaoqiang.li@imj-prg.fr}



\begin{abstract}
By applying a 2014 result on the distribution of full cylinders, we give a proof of the useful folklore: for any $\beta>1$, the Hausdorff dimension of an arbitrary set in the shift space $S_\beta$ is equal to the Hausdorff dimension of its natural projection in $[0,1]$. It has been used in some former papers without proof. Then we clarify that for calculating the Hausdorff dimension of frequency sets using variational formulae, one only needs to focus on the Markov measures of explicit order when the $\beta$-expansion of $1$ is finite. Concretely, it suffices to optimize a function with finitely many variables under some restrictions. Finally, as an application, we obtain an exact formula for the Hausdorff dimension of frequency sets for an important class of $\beta$'s, which are called pseudo-golden ratios (also called multinacci numbers).
\end{abstract}

\maketitle

\section{Introduction}

Let $\beta>1$ be a real number. Given $x\in[0,1)$, the most common way to $\beta$-expand $x$ as
\begin{equation}\label{expansion}
x=\sum_{n=1}^\infty\frac{w_n}{\beta^n}
\end{equation}
is to use the greedy algorithm, which generates (greedy) $\beta$-expansion. It was introduced by R\'enyi \cite{R57} in 1957 and widely studied in the following decades until now \cite{B89,FWL19,FS92,LPWW14,P60,S97,S80}. In some other literature, for examples \cite{ABK19,B18,BK19,DK09,L20,S03}, a $\beta$-expansion of a point $x$ is defined to be a sequence $(w_n)_{n\ge1}$ satisfying (\ref{expansion}). Then a point may have many different $\beta$-expansions including the greedy one. Throughout this paper, we use \textit{$\beta$-expansion} to denote the greedy one defined by the $\beta$-transformation (see Section 2 for definition).

Let $\Sigma_\beta$ be the set of admissible sequences (see Definition \ref{admissible}) and $S_\beta$ be its (topological) closure in the metric space $(\mathcal{A}_\beta^\N,d_\beta)$, where $\mathcal{A}_\beta$ is the alphabet $\{0,1,\cdots,\lceil\beta\rceil-1\}$, $\lceil\beta\rceil$ is the smallest integer no less than $\beta$, $\N$ is the set of positive integers $1,2,3,\cdots$ and $d_\beta$ is the usual metric on $\mathcal{A}_\beta^\N$ defined by
$$d_\beta(w,v):=\beta^{-\inf\{k\ge0: \text{ }w_{k+1}\neq v_{k+1}\}}\quad\text{for }w=(w_n)_{n\ge1},v=(v_n)_{n\ge1}\in\mathcal{A}_\beta^\N.$$
Besides, we use $\pi_\beta:S_\beta\rightarrow[0,1]$ to denote the natural projection map given by
$$\pi_\beta(w):=\frac{w_1}{\beta}+\frac{w_2}{\beta^2}+\cdots+\frac{w_n}{\beta^n}+\cdots \quad \text{for } w=(w_n)_{n\ge1}\in S_\beta.$$

As the first main result in this paper, the following theorem is a folklore result used in some former papers without explicit proof (for example \cite[Section 5]{T12}).

\begin{theorem}\label{dimension equal}
Let $\beta>1$. For any $Z\subset S_\beta$, the Hausdorff dimension of $Z$ in $(S_\beta,d_\beta)$ is equal to the Hausdorff dimension of its natural projection in $[0,1]$, i.e.,
$$\dim_H(Z,d_\beta)=\dim_H\pi_\beta(Z).$$
\end{theorem}

It is worth to note that $\dim_H(Z,d_\beta)\ge\dim_H\pi_\beta(Z)$ follows immediately from the fact that $\pi_\beta$ is Lipschitz continuous. But even if omitting countable many points to make $\pi_\beta$ invertible, the inverse is not Lipschitz continuous. This makes the proof of the inverse inequality much more intricate. We will prove it by using a covering property (see Proposition \ref{coveringproperties}) given by Bugeaud and Wang in 2014 deduced from the distribution of full cylinders.

In the following, we consider the digit frequencies of the expansions. This is a classical research topic began by Borel in 1909. His well known normal number theorem \cite{B09} implies that, for Lebesgue almost every $x\in[0,1]$, the digit frequency of zeros in its binary expansion is equal to $\frac{1}{2}$. Given $\beta>1$, for any $a\in[0,1]$, let $F_{\beta,a}$ be the set of those $x$'s with digit frequencies of $0$'s equal to $a$ in their $\beta$-expansions. That is the \textit{frequency set}
$$F_{\beta,a}:=\Big\{x\in [0, 1): \lim_{n\to\infty}\frac{\#\{1\le k\le n: \varepsilon_k(x,\beta)=0\}}{n}=a\Big\},$$
where $\epsilon_k(x,\beta)$ is the $k$th digit in the $\beta$-expansion of $x$ and $\#$ denotes the cardinality. For $\beta=2$, Borel's normal number theorem means that $F_{2,\frac{1}{2}}$ is of full Lebesgue measure, and implies that $F_{2,a}$ is of zero Lebesgue measure for $a\neq\frac{1}{2}$. This leaves a natural question: How large is $F_{2,a}$ in the sense of dimension? Forty years later, another well known result given by Eggleston \cite{E49} showed that
$$\dim_H F_{2,a}=\frac{-a\log a-(1-a)\log(1-a)}{\log2}\quad\text{for all }a\in[0,1].$$
For the case that $\beta$ is not an integer, the above question, about giving concrete formulae for the Hausdorff dimension of frequency sets, is almost entirely open. Although the Hausdorff dimension of frequency sets can be given by some variational formulae (see for examples \cite{FFW01,TV03,T01}), they are abstract and concrete formulae are very scarce. The only known concrete formula is in \cite[Theorem 1.2]{LLS19}, which contains the well known case that when $\beta=\frac{\sqrt{5}+1}{2}$ is the \textit{golden ratio} (i.e., the $\beta$-expansion of $1$ is $\epsilon(1,\beta)=110^\infty$), we have
$$\dim_HF_{\beta,a}=\frac{a\log a-(2a-1)\log(2a-1)-(1-a)\log(1-a)}{\log\beta}$$
where $\frac{1}{2}\le a\le1$. See for examples \cite{FZ04,LL16}. Note that when $0\le a<\frac{1}{2}$, $F_{\beta,a}=\emptyset$.

As the second main result in this paper, the next theorem takes a step from abstraction to concreteness. It means that for calculating the Hausdorff dimension of frequency sets, we only need to focus on the entropy (see \cite{W82} for definition) with respect to Markov measures of explicit order (see Definition \ref{k-Markov measure}) when $\beta\in(1,2)$ and the $\beta$-expansion of $1$ is finite. More concretely, it suffices to optimize a function with finitely many variables under some restrictions.

For $\beta\in(1,\infty)$, let $\Sigma_\beta^n$ be the set of admissible words with length $n\in\N$ and $\Sigma_\beta^*:=\cup_{n=1}^\infty\Sigma_\beta^n$. For any $w\in\Sigma_\beta^*$, define
$$[w]:=\{v\in S_\beta:v\text{ begins with }w\}$$
to be the cylinder in $S_\beta$ generated by $w$.

Let $\sigma$ be the \textit{shift map} on $\cA_\beta^\N$ defined by
$$\sigma(w_1w_2\cdots):=w_2w_3\cdots \quad \text{for } w \in\cA_\beta^\N.$$
We also use $\sigma$ to denote its restriction on $S_\beta$. Let $\cM_\sigma(S_\beta)$ be the set of $\sigma$-invariant Borel probability measures on $S_\beta$ and $h_\mu(\sigma)$ be the \textit{measure-theoretic entropy} of $\sigma$ with respect to the measure $\mu$. We regard $0\log0$, $0\log\frac{0}{0}$, $\max\emptyset$ and $\sup\emptyset$ as $0$ in the following.

\begin{theorem}\label{Markov variational formula}
Let $\beta\in(1,2)$ such that $\epsilon(1,\beta)=\epsilon_1(1,\beta)\cdots\epsilon_m(1,\beta)0^\infty$ for some integer $m\ge2$ with $\epsilon_m(1,\beta)=1$ and let $a\in[0,1]$. Then
$$\dim_HF_{\beta,a}=\frac{1}{\log\beta}\cdot\max\Big\{h_\mu(\sigma):\mu\in\cM_\sigma(S_\beta),\mu[0]=a, \mu \text{ is an } (m-1) \text{-Markov measure}\Big\}.$$
More concretely,
$$\dim_HF_{\beta,a}=\frac{1}{\log\beta}\cdot\max\Big\{\frak{h}_\mu(\beta,m):\mu\text{ is an }(\beta,m,a)\text{-coordinated set function}\Big\},$$
where for a set function $\mu$ defined from $\{[w]:w\in\cup_{n=1}^m\Sigma_\beta^n\}$ to $[0,1]$,
$$\frak{h}_\mu(\beta,m):=-\sum_{w_1\cdots w_m\in\Sigma_\beta^m}\mu[w_1\cdots w_m]\log\frac{\mu[w_1\cdots w_m]}{\mu[w_1\cdots w_{m-1}]},$$
and $\mu$ is called \textit{$(\beta,m,a)$-coordinated} if
$$\mu[0]=a,\text{ }\mu[1]=1-a,\quad\sum_{\substack{v=0,1\\wv\in\Sigma_\beta^*}}\mu[wv]=\mu[w]\quad\text{and}\quad\sum_{\substack{u=0,1\\uw\in\Sigma_\beta^*}}\mu[uw]=\mu[w]$$
for all $w\in\cup_{n=1}^{m-1}\Sigma_\beta^n$.
\end{theorem}

Note that for any $(m-1)$-Markov measure $\mu\in\cM_\sigma(S_\beta)$, $h_\mu(\sigma)$ is exactly equal to $\frak{h}_\mu(\beta,m)$ (see Proposition \ref{Markov entropy}).

As an application of the above theorem, in the following we give an exact formula for the Hausdorff dimension of the frequency sets for an important class of $\beta$'s, which are called \textit{pseudo-golden ratios}.

\begin{theorem}\label{pseudo-golden ratios} Let $\beta\in(1,2)$ such that $\epsilon(1,\beta)=1^m0^\infty$ for some integer $m\ge3$.
\newline\emph{(1)} If $0\le a<\frac{1}{m}$, then $F_{\beta,a}=\emptyset$ and $\dim_H F_{\beta,a}=0$.
\newline\emph{(2)} If $\frac{1}{m}\le a\le1$, then
$$\dim_H F_{\beta,a}=\frac{1}{\log\beta}\cdot\max_{x_1,\cdots,x_{m-2}} f_a(x_1,\cdots,x_{m-2})$$
where $f_a(x_1,\cdots,x_{m-2})$
$$\begin{aligned}
=a\log a&-(a-x_1)\log(a-x_1)\\
&-(x_1-x_2)\log(x_1-x_2)\\
&\cdots\\
&-(x_{m-3}-x_{m-2})\log(x_{m-3}-x_{m-2})\\
&-(1-a-x_1-\cdots-x_{m-2})\log(1-a-x_1-\cdots-x_{m-2})\\
&-(x_1+\cdots+x_{m-3}+2x_{m-2}+a-1)\log(x_1+\cdots+x_{m-3}+2x_{m-2}+a-1)
\end{aligned}$$
and the maximum is taken over $x_1,\cdots,x_{m-2}$ such that all terms in the $\log$'s are non-negative. That is, $a\ge x_1\ge x_2\ge\cdots\ge x_{m-2}\ge0$ and $x_1+\cdots+x_{m-3}+x_{m-2}\le 1-a \le x_1+\cdots+x_{m-3}+2x_{m-2}$.

In particular, $\dim_H F_{\beta,\frac{1}{m}}=\dim_H F_{\beta,1}=0$.
\end{theorem}

\begin{remark} For the case $m=3$, i.e., $\epsilon(1,\beta)=1110^\infty$, given any $a\in[\frac{1}{3},1]$, by calculating the derivative of $f_a(x_1)$, it is straightforward to get
$$\begin{small}\begin{aligned}
dim_H F_{\beta,a}=\frac{1}{\log\beta}\Big(a\log a&-\frac{10a-3-\sqrt{-8a^2+12a-3}}{6}\log\frac{10a-3-\sqrt{-8a^2+12a-3}}{6}\\
&-\frac{-2a+3-\sqrt{-8a^2+12a-3}}{6}\log\frac{-2a+3-\sqrt{-8a^2+12a-3}}{6}\\
&-\frac{-a+\sqrt{-8a^2+12a-3}}{3}\log\frac{-a+\sqrt{-8a^2+12a-3}}{3}\Big).
\end{aligned}\end{small}$$
In particular, $\dim_H F_{\beta,\frac{1}{3}}=\dim_H F_{\beta,1}=0$.
\end{remark}

We introduce some preliminaries in the next section, and then prove Theorems \ref{dimension equal}, \ref{Markov variational formula} and \ref{pseudo-golden ratios} in Sections 3, 4 and 5 respectively.

\section{Preliminaries}

For $\beta>1$, we define the \textit{$\beta$-transformation} $T_\beta:[0,1]\rightarrow[0,1]$ by
$$T_\beta(x):=\beta x-\lfloor\beta x \rfloor\quad\text{for } x\in[0,1]$$
where $\lfloor y \rfloor$ denotes the greatest integer no larger than $y$. For any $n\in\N$ and $x\in[0,1]$, define
$$\epsilon_n(x,\beta) := \lfloor \beta T^{n-1}_\beta (x)\rfloor.$$
Then we can write
$$x = \sum_{n = 1}^\infty\frac{\epsilon_n(x,\beta)}{\beta^n}$$
and we call the sequence $\epsilon(x,\beta):=\epsilon_1(x,\beta)\epsilon_2(x,\beta)\cdots\epsilon_n(x,\beta)\cdots$ the $\beta$-\textit{expansion} of $x$. Besides, the sequence $\epsilon(x,\beta)$ is said to be infinite if there are infinitely many $n\in\N$ such that $\epsilon_n(x,\beta)\neq 0$. Otherwise, there exists a smallest $m\in \mathbb{N}$ such that for any $j>m, \epsilon_j(x,\beta)=0$ but $\epsilon_m(x,\beta)\neq0$, and we say that $\epsilon(x,\beta)$ is finite with length $m$.

The \textit{quasi-greedy} $\beta$-expansion of $1$ defined by
$$\epsilon^*(1,\beta):=\left\{\begin{array}{ll}
\epsilon(1,\beta) & \mbox{if } \epsilon(1,\beta) \mbox{ is infinite}\\
(\epsilon_1(1,\beta)\cdots\epsilon_{m-1}(1,\beta)(\epsilon_m(1,\beta)-1))^\infty & \mbox{if } \epsilon(1,\beta) \mbox{ is finite with length } m
\end{array}\right.$$
is very useful for checking the admissibility of a sequence (see Lemma \ref{charADM}).

Recall that $\mathcal{A}_\beta$ is the alphabet $\{0,1,\cdots,\lceil\beta\rceil-1\}$ and $d_\beta$ is the usual metric on $\mathcal{A}_\beta^\N$.

\begin{definition}[Admissibility]\label{admissible} Let $\beta>1$. A sequence $w \in \mathcal{A}_\beta^\N$ is called \textit{admissible} if there exists $x \in [0,1)$ such that $\epsilon_i(x,\beta) = w_i$ for all $i \in \N$. We denote the set of all admissible sequences by $\Sigma_\beta$ and its closure in $(\cA_\beta^\N,d_\beta)$ by $S_\beta$. For $n\in\N$, a word $w \in \mathcal{A}_\beta^n$ is called \textit{admissible} if there exists $x \in [0,1)$ such that $\epsilon_i(x,\beta) = w_i$ for all $i\in\{1,\cdots,n\}$. We denote the set of all admissible words with length $n$ by $\Sigma_\beta^n$ and write
$$\Sigma_\beta^*:=\bigcup_{n=1}^\infty\Sigma_\beta^n.$$
\end{definition}

One can verify that $\sigma(S_\beta)=S_\beta$. When we write $\sigma^{-1}$, we consider $\sigma$ restricted to $S_\beta$. So $\sigma^{-1}A\subset S_\beta$ for all $A\subset S_\beta$.

The following criterion due to Parry is well known.

\begin{lemma}[\cite{P60}]\label{charADM}
Let $\beta>1$ and $w$ be a sequence in $\mathcal{A}_\beta^\N$. Then
$$w \in \Sigma_\beta\quad\Longleftrightarrow\quad\sigma^k (w) \prec \epsilon^*(1,\beta) \quad \text{for all } k \ge 0$$
and
$$w \in S_\beta\quad\Longleftrightarrow\quad\sigma^k (w) \preceq \epsilon^*(1,\beta) \quad \text{for all } k \ge 0$$
where $\prec$ and $\preceq$ denote the \textit{lexicographic order} in $\cA_\beta^\N$.
\end{lemma}

We prove the following useful proposition.

\begin{proposition}\label{every m}
Let $\beta>1$ such that $\epsilon(1,\beta)=\epsilon_1(1,\beta)\cdots\epsilon_m(1,\beta)0^m$ for some integer $m\ge2$ with $\epsilon_m(1,\beta)\neq0$ and $w_1\cdots w_n\in\mathcal{A}_\beta^n$ for some integer $n\ge m$, then
$$w_1\cdots w_n\in\Sigma_\beta^*\quad\text{if and only if}\quad w_1\cdots w_m,w_2\cdots w_{m+1},\cdots,w_{n-m+1}\cdots w_n\in\Sigma_\beta^*.$$
\end{proposition}
\begin{proof} \boxed{\Rightarrow} Obvious.
\newline\boxed{\Leftarrow} For simplification we use $\epsilon_1,\cdots,\epsilon_m$ instead of $\epsilon_1(1,\beta),\cdots,\epsilon_m(1,\beta)$ in the following. Suppose
$$w_1\cdots w_m,w_2\cdots w_{m+1},\cdots,w_{n-m+1}\cdots w_n\in\Sigma_\beta^*.$$
By Lemma \ref{charADM} we get
$$w_1\cdots w_m,w_2\cdots w_{m+1},\cdots,w_{n-m+1}\cdots w_n\preceq\epsilon_1\cdots\epsilon_{m-1}(\epsilon_m-1).$$
In order to get $w_1\cdots w_n\in\Sigma_\beta^*$, by Lemma \ref{charADM}, it suffices to check
$$\sigma^k(w_1\cdots w_n0^\infty)\prec(\epsilon_1\cdots\epsilon_{m-1}(\epsilon_m-1))^\infty\quad\text{for all }k\ge0.$$
If $k\ge n$, this is obvious. We consider $k\le n-1$ in the following. Let $l\ge0$ be the greatest integer such that $k+lm\le n-1$. Then
$$\begin{aligned}
&\sigma^k(w_1\cdots w_n0^\infty)\\
=&(w_{k+1}\cdots w_{k+m})(w_{k+m+1}\cdots w_{k+2m})\cdots(w_{k+(l-1)m+1}\cdots w_{k+lm})(w_{k+lm+1}\cdots w_n0^{k+(l+1)m-n})0^\infty\\
\preceq&(\epsilon_1\cdots\epsilon_{m-1}(\epsilon_m-1))^l(w_{k+lm+1}\cdots w_n0^{k+(l+1)m-n})0^\infty\\
\prec&(\epsilon_1\cdots\epsilon_{m-1}(\epsilon_m-1))^\infty,
\end{aligned}$$
where the last inequality follows from
\begin{eqnarray}\label{tail less}
w_{k+lm+1}\cdots w_n0^{k+(l+1)m-n}\preceq\epsilon_1\cdots\epsilon_{m-1}(\epsilon_m-1),
\end{eqnarray}
which can be proved as follows. In fact, by $w_{n-m+1}\cdots w_n\in\Sigma_\beta^*$ and Lemma \ref{charADM}, we get
$$\sigma^{k+(l+1)m-n}(w_{n-m+1}\cdots w_n0^\infty)\prec(\epsilon_1\cdots\epsilon_{m-1}(\epsilon_m-1))^\infty.$$
This implies (\ref{tail less}).
\end{proof}

In this paper, we use the following definitions of cylinders, noting that in some literature $[w]$ denotes the cylinder in $\Sigma_\beta$, not in $S_\beta$.

\begin{definition}[Cylinder]\label{Cylinder} Let $\beta>1$. For an admissible word $w\in\Sigma_\beta^*$ with length $n\in\N$, the \textit{cylinder} in $S_\beta$ of order $n$ generated by $w$ is defined by
$$[w]:=\{v\in S_\beta:v_i=w_i\text{ for }1\le i\le n\},$$
and the \textit{cylinder} in $[0,1)$ of order $n$ generated by $w$ is defined by
$$I(w):=\{x\in[0,1):\epsilon_i(x,\beta)=w_i\text{ for }1\le i\le n\}.$$
\end{definition}

The following covering property, which plays a crucial role in the proof of Theorem \ref{dimension equal}, is deduced from the length and distribution of full cylinders (see \cite{BW14,FW12,LL18} for definition and more details).

\begin{proposition}(\cite[Proposition 4.1]{BW14})\label{coveringproperties}
Let $\beta>1$. For any $x\in[0,1)$ and $n\in\N$, the interval $[x-\frac{1}{\beta^n},x+\frac{1}{\beta^n}]$ intersected with $[0,1)$ can be covered by at most $4(n+1)$ cylinders of order $n$.
\end{proposition}

\begin{definition}[Hausdorff measure and dimension in metric space]\label{dim in metric space} Let $(X,d)$ be a metric space. For any $U\subset X$, denote the diameter of $U$ by $|U|:=\sup_{x,y\in U}d(x,y)$. For any $A\subset X, s\ge0$ and $\delta>0$, let
$$\cH^s_\delta(A,d):=\inf\Big\{\sum_{i=1}^\infty|U_i|^s:A\subset\bigcup_{i=1}^\infty U_i\text{ and }|U_i|\le\delta\text{ for all }i\in\N\Big\}.$$
We define the \textit{$s$-dimensional Hausdorff measure} of $A$ in $(X,d)$ by
$$\cH^s(A,d):=\lim_{\delta\to0}\cH^s_\delta(A,d)$$
and the \textit{Hausdorff dimension} of $A$ in $(X,d)$ by
$$\dim_H(A,d):=\sup\{s\ge0:\cH^s(A,d)=\infty\}.$$
In the space of real numbers $\R$ (equipped with the usual metric), we use $\cH^s(A)$ and $\dim_H A$ to denote the $s$-dimensional Hausdorff measure and the Hausdorff dimension of $A$ respectively for simplification (see \cite{F90}).
\end{definition}

\begin{definition}[Lipschitz continuous]
Let $(X,d)$ and $(X',d')$ be two metric spaces. A map $f:X\to X'$ is called \textit{Lipschitz continuous} if there exists a constant $c>0$ such that
$$d'(f(x),f(y))\le c\cdot d(x,y)\quad\text{for all }x,y\in X.$$
\end{definition}

The following basic proposition can be deduced directly from the definitions.

\begin{proposition}\label{lip}
If the map $f:(X,d)\to(X',d')$ between two metric spaces is Lipschitz continuous, then for any $A\subset X$, we have
$$\dim_H(f(A),d')\le\dim_H(A,d).$$
\end{proposition}

Recall that $\cM_\sigma(S_\beta)$ is the set of $\sigma$-invariant Borel probability measures on $S_\beta$. The following is a consequence of Carath\'eodory's measure extension theorem and the fact that for verifying the $\sigma$-invariance of measures on $S_\beta$, one only needs to check it for the cylinders.

\begin{proposition}\label{extend} Let $\beta\in(1,2]$. Any set function $\mu$ from $\{[w]:w\in\Sigma_\beta^*\}$ to $[0,1]$ satisfying
$$\mu[0]+\mu[1]=1,\quad\sum_{\substack{v=0,1\\wv\in\Sigma_\beta^*}}\mu[wv]=\mu[w]\quad\text{and}\quad\sum_{\substack{u=0,1\\uw\in\Sigma_\beta^*}}\mu[uw]=\mu[w]$$
for all $w\in\Sigma_\beta^*$ can be uniquely extended to be a measure in $\cM_\sigma(S_\beta)$.
\end{proposition}

The following concept is well known (see for examples \cite[Section 2]{FFW01} and \cite[Section 6.2]{K97}).

\begin{definition}[$k$-Markov measure]\label{k-Markov measure} Let $\beta\in(1,2]$, $k\in\N$ and $\mu\in\cM_\sigma(S_\beta)$. We call $\mu$ a \textit{$k$-Markov measure} if
$$\mu[w_1\cdots w_n]=\mu[w_1\cdots w_{n-1}]\cdot\frac{\mu[w_{n-k}\cdots w_n]}{\mu[w_{n-k}\cdots w_{n-1}]}$$
for all $w_1\cdots w_n\in\Sigma_\beta^n$ with $n>k$.
\end{definition}

Recall that $h_\mu(\sigma)$ is the measure-theoretic entropy of $\sigma$ with respect to the measure $\mu$. Using $\cP:=\{[0],[1]\}$ as a partition generator of the Borel sigma-algebra on $S_\beta$, the proof of the following proposition is straightforward.

\begin{proposition}\label{Markov entropy} Let $\beta\in(1,2]$, $k\in\N$ and $\mu\in\cM_\sigma(S_\beta)$ be a $k$-Markov measure, then
$$h_\mu(\sigma)=-\sum_{w_1\cdots w_{k+1}\in\Sigma_\beta^{k+1}}\mu[w_1\cdots w_{k+1}]\log\frac{\mu[w_1\cdots w_{k+1}]}{\mu[w_1\cdots w_k]}.$$
\end{proposition}

\section{Proof of Theorem \ref{dimension equal}}

The main we need to prove is the following technical lemma.

\begin{lemma}\label{measurescompare}
Let $\beta>1$, $s>0$ and $Z\subset S_\beta$. Then for any $\epsilon\in(0,s)$, we have
$$\cH^s(Z,d_\beta)\le\cH^{s-\epsilon}(\pi_\beta(Z)).$$
\end{lemma}
\begin{proof} Fix $\epsilon\in(0,s)$. Let $Z_0:=Z\cap\Sigma_\beta$. Since $S_\beta\setminus\Sigma_\beta$ is countable, we only need to prove $\cH^s(Z_0,d_\beta)\le\cH^{s-\epsilon}(\pi_\beta(Z_0))$.
\newline(1) Choose $\delta_0\in(0,\frac{1}{\beta})$ small enough as follows. Since $\beta^{(n+1)\epsilon}\to\infty$ much faster than $8\beta^sn\to\infty$ as $n\to\infty$, there exists $n_0\in\N$ such that for any $n>n_0$, $8\beta^s n\le\beta^{(n+1)\epsilon}$. By $\frac{-\log\delta}{\log\beta}-1\to\infty$ as $\delta\to0^+$, there exists $\delta_0\in(0,\frac{1}{\beta})$ small enough such that $\frac{-\log\delta_0}{\log\beta}-1>n_0$. Then for any $n>\frac{-\log\delta_0}{\log\beta}-1$, we will have $8\beta^sn\le\beta^{(n+1)\epsilon}$.
\newline(2) For any $\delta\in(0,\delta_0)$, let $\{U_i\}$ be a $\delta$-cover of $\pi_\beta(Z_0)$, i.e., $0<|U_i|\le\delta$ and $\pi_\beta(Z_0)\subset\cup_i U_i$. Then for each $U_i$, there exists $n_i\in\N$ such that $\frac{1}{\beta^{n_i+1}}<|U_i|\le\frac{1}{\beta^{n_i}}$. By Proposition \ref{coveringproperties}, $U_i$ can be covered by at most $8n_i$ cylinders $I_{i,1}, I_{i,2}, \cdots, I_{i,8n_i}$ of order $n_i$. It follows from
$$|\Sigma_\beta\cap\pi_\beta^{-1}I_{i,j}|=\frac{1}{\beta^{n_i}}<\beta|U_i|\le\beta\delta\quad\text{and}\quad Z_0\subset\Sigma_\beta\cap\bigcup_i\pi_\beta^{-1}U_i\subset\bigcup_i\bigcup_{j=1}^{8n_i}(\Sigma_\beta\cap\pi_\beta^{-1}I_{i,j})$$
that
\begin{equation}\label{upper-z0}
\cH^s_{\beta\delta}(Z_0,d_\beta)\le\sum_i\sum_{j=1}^{8n_i}|\Sigma_\beta\cap\pi_\beta^{-1}I_{i,j}|^s=\sum_i\frac{8n_i}{\beta^{n_is}}\overset{(\star)}{\le}\sum_i\frac{1}{\beta^{(n_i+1)(s-\epsilon)}}<\sum_i|U_i|^{s-\epsilon},
\end{equation}
where ($\star$) is because $\frac{1}{\beta^{n_i+1}}<|U_i|<\delta_0$ implies $n_i>\frac{-\log\delta_0}{\log\beta}-1$, and then by (1) we have $8n_i\beta^s\le\beta^{(n_i+1)\epsilon}$. Taking $\inf$ on the right of (\ref{upper-z0}), we get $\cH^s_{\beta\delta}(Z_0,d_\beta)\le\cH^{s-\epsilon}_\delta(\pi_\beta(Z_0))$. It follows from letting $\delta\to0$ that $\cH^s(Z_0,d_\beta)\le\cH^{s-\epsilon}(\pi_\beta(Z_0))$.
\end{proof}

\begin{proof}[Proof of Theorem \ref{dimension equal}]
The inequality $\dim_H(Z,d_\beta)\ge\dim_H\pi_\beta(Z)$ follows from Proposition \ref{lip} and the fact that $\pi_\beta$ is Lipschitz continuous. The inverse inequality follows from Lemma \ref{measurescompare}. In fact, for any $t<\dim_H(Z,d_\beta)$, there exists $s$ such that $t<s<\dim_H(Z,d_\beta)$. By $\cH^s(Z,d_\beta)=\infty$ and Lemma \ref{measurescompare}, we get $\cH^t(\pi_\beta(Z))=\infty$. Thus $t\le\dim_H\pi_\beta(Z)$. It means that $\dim_H(Z,d_\beta)\le\dim_H\pi_\beta(Z)$.
\end{proof}

\section{Proof of Theorem \ref{Markov variational formula}}

We will deduce Theorem \ref{Markov variational formula} from the following proposition, which is essentially from \cite{PS05}.

\begin{proposition}\label{variational formula} Let $\beta>1$ and $a\in[0,1]$. Then
$$\dim_HF_{\beta,a}=\frac{1}{\log\beta}\cdot\sup\Big\{h_{\mu}(\sigma):\mu\in\cM_\sigma(S_\beta),\mu[0]=a\Big\}.$$
\end{proposition}

For the convenience of the readers, we recall some definitions and show how Proposition \ref{variational formula} comes from \cite{PS05}.

\begin{definition}\label{psDef} Let $\beta>1$.
\item(1) For any $w\in S_\beta$ and $n\in\N$, the \textit{empirical measure} is defined by
$$\cE_n(w):=\frac{1}{n}\sum_{i=0}^{n-1}\delta_{\sigma^iw}$$
where $\delta_w$ is the Dirac probability measure concentrated on $w$.
\item(2) Let $\cA$ be an arbitrary non-empty parameter set and let
$$\cF:=\Big\{(f_\alpha,c_\alpha,d_\alpha):\alpha\in\cA\Big\}$$
where $f_\alpha:S_\beta\to\R$ is continuous and $c_\alpha,d_\alpha\in\R$ with $c_\alpha\le d_\alpha$ for all $\alpha\in\cA$. Define
$$S_{\beta,\cF}:=\Big\{w\in S_\beta:\forall\alpha\in\cA,c_\alpha\le\varliminf_{n\to\infty}\int f_\alpha\text{ } d\cE_n(w)\le\varlimsup_{n\to\infty}\int f_\alpha\text{ }d\cE_n(w)\le d_\alpha\Big\}$$
and
$$\cM_{\beta,\cF}:=\Big\{\mu\in\cM_\sigma(S_\beta):\forall\alpha\in\cA, c_\alpha\le\int f_\alpha\text{ } d\mu\le d_\alpha\Big\}.$$
\end{definition}

Combining Theorems 5.2 and 5.3 in \cite{PS05}, we get the following.

\begin{lemma}\label{psThm}
Let $\beta>1$. If $\cM_{\beta,\cF}$ is a non-empty closed connected set, then
$$h_{top}(S_{\beta,\cF},\sigma)=\sup\Big\{h_\mu(\sigma):\mu\in\cM_{\beta,\cF}\Big\}$$
where $h_{top}(S_{\beta,\cF},\sigma)$ is the topological entropy of $S_{\beta,\cF}$ in the dynamical system $(S_\beta,d_\beta,\sigma)$. (See \cite{B73} for the definition of the \textit{topological entropy} for non-compact sets.)
\end{lemma}

For $\beta>1$ and $a\in[0,1]$, let
$$S_{\beta,a}:=\Big\{w\in S_\beta:\lim_{n\to\infty}\frac{\#\{1\le k\le n:w_k=0\}}{n}=a\Big\}.$$
In Definition \ref{psDef} (2), let $\cF$ be the singleton $\{(\mathbbm{1}_{[0]},a,a)\}$, where the characteristic function $\mathbbm{1}_{[0]}:S_\beta\to\R$ is continuous. (Here we note that another characteristic function $\mathbbm{1}_{[0,\frac{1}{\beta}]}:[0,1]\to\R$ is not continuous, which means that some other similar variational formulae corresponding to dynamical systems on [0,1] can not be applied directly in our case.) We get the following lemma as a special case of the above one.

\begin{lemma}\label{topS}
$$h_{top}(S_{\beta,a},\sigma)=\sup\Big\{h_\mu(\sigma):\mu\in\cM_\sigma(S_\beta),\mu[0]=a\Big\}.$$
\end{lemma}

Hence, Proposition \ref{variational formula} follows from
\begin{eqnarray*}
\dim_H F_{\beta,a}&\xlongequal[\text{is countable}]{\pi_\beta(S_{\beta,a})\setminus F_{\beta,a}}&\dim_H \pi_\beta(S_{\beta,a})\\
&\xlongequal[\text{Theorem \ref{dimension equal}}]{\text{by}}&\dim_H(S_{\beta,a},d_\beta)\\
&\xlongequal[\text{Lemma \ref{dim-top}}]{\text{by}}&\frac{1}{\log\beta}\cdot h_{top}(S_{\beta,a},\sigma),
\end{eqnarray*}
where $\pi_\beta(S_{\beta,a})\setminus F_{\beta,a}$ is countable since we can check $\pi_\beta(S_{\beta,a})\setminus F_{\beta,_a}\subset\pi_\beta(S_\beta\setminus\Sigma_\beta)$ and Lemma \ref{charADM} implies that $S_\beta\setminus\Sigma_\beta$ is countable.

\begin{lemma}(\cite[Lemma 5.3]{T12})\label{dim-top} Let $\beta>1$. For any $Z\subset S_\beta$, we have $$\dim_H(Z,d_\beta)=\frac{1}{\log\beta}\cdot h_{top}(Z,\sigma).$$
\end{lemma}

We give the following proofs to end this section.

\begin{proof}[Proof of Lemma \ref{topS}]
In Definition \ref{psDef} (2), let $\cF$ be the singleton $\{(\mathbbm{1}_{[0]},a,a)\}$. Then
$$S_{\beta,\cF}=\Big\{w\in S_\beta:\lim_{n\to\infty}\frac{1}{n}\sum_{i=0}^{n-1}\mathbbm{1}_{[0]}(\sigma^iw)=a\Big\}=S_{\beta,a}$$
and
$$\cM_{\beta,\cF}=\Big\{\mu\in\cM_\sigma(S_\beta):\mu[0]=a\Big\}\xlongequal[\text{by}]{\text{denote}}:\cM_{\beta,a}.$$
\item(1) If $\cM_{\beta,a}=\emptyset$, we can prove $S_{\beta,a}=\emptyset$ (and then the conclusion follows).
\newline(By contradiction) If $S_{\beta,a}\neq\emptyset$, there exists $w\in S_{\beta,a}$. For any $n\in\N$, let
$$\mu_n:=\cE_n(w)\in\cM(S_\beta):=\{\text{Borel probability measures on } S_\beta\}.$$
Since $\cM(S_\beta)$ is compact, there exists subsequence $\{\mu_{n_k}\}_{k\in\N}\subset\{\mu_n\}_{n\in\N}$ and $\mu\in\cM(S_\beta)$ such that $\mu_{n_k}\overset{w^*}{\to}\mu$ (i.e. $\mu_{n_k}$ converge to $\mu$ under the weak* topology). By $\mu_{n_k}\circ\sigma^{-1}\overset{w^*}{\to}\mu\circ\sigma^{-1}$ and $\mu_{n_k}\circ\sigma^{-1}-\mu_{n_k}\overset{w^*}{\to}0$, we get $\mu\circ\sigma^{-1}=\mu$ and then $\mu\in\cM_\sigma(S_\beta)$. It follows from
$$\mu[0]=\int\mathbbm{1}_{[0]}\text{ }d\mu=\lim_{k\to\infty}\int\mathbbm{1}_{[0]}\text{ }d\mu_{n_k}=\lim_{k\to\infty}\frac{1}{n_k}\sum_{i=0}^{n_k-1}\mathbbm{1}_{[0]}(\sigma^iw)\xlongequal[]{w\in S_{\beta,a}}a$$
that $\mu\in\cM_{\beta,a}$, which contradicts $\cM_{\beta,a}=\emptyset$.
\item(2) If $\cM_{\beta,a}\neq\emptyset$, by Lemma \ref{psThm}, it suffices to prove that $\cM_{\beta,a}$ is a closed connected set in $\cM_\sigma(S_\beta)$.
\begin{itemize}
\item[\textcircled{\footnotesize{$1$}}] Prove that $\cM_{\beta,a}$ is closed.
  \newline Let $\{\mu_n,n\in\N\}\subset\cM_{\beta,a}$ and $\mu\in \cM_\sigma(S_\beta)$ such that $\mu_n\overset{w^*}{\to}\mu$. It follows from
  $$\mu[0]=\int\mathbbm{1}_{[0]}\text{ }d\mu=\lim_{n\to\infty}\int\mathbbm{1}_{[0]}\text{ }d\mu_n=\lim_{n\to\infty}\mu_n[0]=a$$
  that $\mu\in\cM_{\beta,a}$.
\item[\textcircled{\footnotesize{$2$}}] Prove that $\cM_{\beta,a}$ is connected.
  \newline It suffices to prove that $\cM_{\beta,a}$ is path connected. In fact, for any $\mu_0,\mu_1\in\cM_{\beta,a}$, we define the path $f:[0,1]\to\cM_{\beta,a}$ by $f(s):=\mu_s:=(1-s)\mu_0+s\mu_1$ for $s\in[0,1]$. Then $f(0)=\mu_0$, $f(1)=\mu_1$ and $f([0,1])\subset\cM_{\beta,a}$. It remains to show that $f$ is continuous. Let $\{s,s_n,n\ge1\}\subset[0,1]$ such that $s_n\to s$. We only need to prove $f(s_n)\to f(s)$, i.e., $\mu_{s_n}\overset{w^*}{\to}\mu_s$. Let $\phi:S_\beta\to\R$ be a continuous function. It suffices to check $\int\phi$ $d\mu_{s_n}\to\int\phi$ $d\mu_s$, i.e.,
  $$(1-s_n)\int\phi\text{ }d\mu_0+s_n\int\phi\text{ }d\mu_1\to(1-s)\int\phi\text{ }d\mu_0+s\int\phi\text{ }d\mu_1.$$
  This follows immediately from $s_n\to s$.
\end{itemize}
\end{proof}

\begin{proof}[Proof of Theorem \ref{Markov variational formula}] By Proposition \ref{variational formula} it suffices to consider the following (1), (2) and (3).
\newline(1) We have
\begin{eqnarray*}
& &\sup\Big\{h_\mu(\sigma):\mu\in\cM_\sigma(S_\beta),\mu[0]=a,\mu\text{ is an }(m-1)\text{-Markov measure}\Big\}\\
&\le&\sup\Big\{h_\mu(\sigma):\mu\in\cM_\sigma(S_\beta),\mu[0]=a\Big\}\\
&\le&\sup\Big\{\frak{h}_\mu(\beta,m):\mu\text{ is an }(\beta,m,a)\text{-coordinated set function}\Big\}.
\end{eqnarray*}
Since the first inequality is obvious, we only prove the second one as follows. Let $\mu\in\cM_\sigma(S_\beta)$ such that $\mu[0]=a$. Restricted to $\{[w]:w\in\cup_{n=1}^m\Sigma_\beta^n\}$, $\mu$ is obviously an $(\beta,m,a)$-coordinated set function. It suffices to prove $h_\mu(\sigma)\le\frak{h}_\mu(\beta,m)$. Using $\cP:=\{[0],[1]\}$ as a partition generator of the Borel sigma-algebra on $(S_\beta,d_\beta)$, by simple calculation, we get that the conditional entropy of $\cP$ given $\bigvee_{k=1}^{m-1}\sigma^{-k}\cP$ with respect to $\mu$, denoted by $H_\mu\Big(\cP\mid\bigvee_{k=1}^{m-1}\sigma^{-k}\cP\Big)$, is equal to $\frak{h}_\mu(\beta,m)$. Since $H_\mu\Big(\cP\mid\bigvee_{k=1}^{n-1}\sigma^{-k}\cP\Big)$ decreases as $n$ increases and \cite[Theorem 4.14]{W82} says that it converges to $h_\mu(\sigma)$, we get $h_\mu(\sigma)\le\frak{h}_\mu(\beta,m)$. In the following we attached the calculation mentioned above.
\begin{align*}
&H_\mu\Big(\cP\mid\bigvee_{k=1}^{m-1}\sigma^{-k}\cP\Big)=H_\mu\Big(\cP\mid\sigma^{-1}(\bigvee_{k=0}^{m-2}\sigma^{-k}\cP)\Big)\\
=&-\sum_{P\in\cP,\text{ }Q\in\bigvee_{k=0}^{m-2}\sigma^{-k}\cP}\mu(P\cap\sigma^{-1}Q)\log\frac{\mu(P\cap\sigma^{-1}Q)}{\mu(\sigma^{-1}Q)}\\
=&-\sum_{w_1\cdots w_m\in\Sigma_\beta^*}\mu[w_1\cdots w_m]\log\frac{\mu[w_1\cdots w_m]}{\mu(\sigma^{-1}[w_2\cdots w_m])}\\
=&\sum_{w_1\cdots w_m\in\Sigma_\beta^*}\mu[w_1\cdots w_m]\log\mu[w_2\cdots w_m]-\sum_{w_1\cdots w_m\in\Sigma_\beta^*}\mu[w_1\cdots w_m]\log\mu[w_1\cdots w_m]\\
=&\sum_{w_2\cdots w_m\in\Sigma_\beta^*}\mu[w_2\cdots w_m]\log\mu[w_2\cdots w_m]-\sum_{w_1\cdots w_m\in\Sigma_\beta^*}\mu[w_1\cdots w_m]\log\mu[w_1\cdots w_m]\\
=&\sum_{w_1\cdots w_{m-1}\in\Sigma_\beta^*}\mu[w_1\cdots w_{m-1}]\log\mu[w_1\cdots w_{m-1}]-\sum_{w_1\cdots w_m\in\Sigma_\beta^*}\mu[w_1\cdots w_m]\log\mu[w_1\cdots w_m]\\
=&\sum_{w_1\cdots w_m\in\Sigma_\beta^*}\mu[w_1\cdots w_m]\log\mu[w_1\cdots w_{m-1}]-\sum_{w_1\cdots w_m\in\Sigma_\beta^*}\mu[w_1\cdots w_m]\log\mu[w_1\cdots w_m]\\
=&-\sum_{w_1\cdots w_m\in\Sigma_\beta^*}\mu[w_1\cdots w_m]\log\frac{\mu[w_1\cdots w_m]}{\mu[w_1\cdots w_{m-1}]}=\frak{h}_\mu(\beta,m).
\end{align*}
(2) Prove
$$\begin{aligned}
&\Big\{h_\mu(\sigma):\mu\in\cM_\sigma(S_\beta),\mu[0]=a,\mu\text{ is an }(m-1)\text{-Markov measure}\Big\}\\
&=\Big\{\frak{h}_\mu(\beta,m):\mu\text{ is an }(\beta,m,a)\text{-coordinated set function}\Big\}.
\end{aligned}$$
\boxed{\subset} follows from the facts that every $(m-1)$-Markov measure $\mu\in M_\sigma(S_\beta)$ with $\mu[0]=a$ restricted to $\{[w]:w\in\cup_{n=1}^m\Sigma_\beta^n\}$ is an $(\beta,m,a)$-coordinated set function and Proposition \ref{Markov entropy} implies $h_\mu(\sigma)=\frak{h}_\mu(\beta,m)$.
\newline\boxed{\supset} Let $\mu$ be an $(\beta,m,a)$-coordinated set function. By the entropy formula Proposition \ref{Markov entropy}, it suffices to show that $\mu$ can be extended to be an $(m-1)$-Markov measure in $\cM_\sigma(S_\beta)$. Note that $\mu$ is already defined on all the cylinders of order $\le m$. Suppose that for some $n\ge m$, $\mu$ is already defined on $\{[w_1\cdots w_n]:w_1\cdots w_n\in\Sigma_\beta^n\}$. Then we define
$$\mu[w_1\cdots w_{n+1}]:=\mu[w_1\cdots w_n]\cdot\frac{\mu[w_{n-m+2}\cdots w_{n+1}]}{\mu[w_{n-m+2}\cdots w_n]}$$
where the right hand side is regarded as $0$ if one of $\mu[w_1\cdots w_n]$, $\mu[w_{n-m+2}\cdots w_n]$ and $\mu[w_{n-m+2}\cdots w_{n+1}]$ is $0$. By Proposition \ref{extend} it suffices to check
$$\text{\textcircled{\footnotesize{1}} }\sum_{\substack{v=0,1\\wv\in\Sigma_\beta^*}}\mu[wv]=\mu[w]\quad\text{and}\quad\text{\textcircled{\footnotesize{2}} }\sum_{\substack{u=0,1\\uw\in\Sigma_\beta^*}}\mu[uw]=\mu[w]$$
for all $w\in\Sigma_\beta^n$ with $n\ge m$. (Note that for $n\le m-1$, \textcircled{\footnotesize{1}} and \textcircled{\footnotesize{2}} are already guaranteed by the condition that $\mu$ is $(\beta,m,a)$-coordinated.)
\newline\textcircled{\footnotesize{1}} Let $n\ge m$ and $w_1\cdots w_n\in\Sigma_\beta^n$.
\begin{itemize}
\item[i)] If $w_1\cdots w_n1\in\Sigma_\beta^*$, then
$$\begin{aligned}
&\sum_{\substack{v=0,1\\w_1\cdots w_nv\in\Sigma_\beta^*}}\mu[w_1\cdots w_nv]=\mu[w_1\cdots w_n0]+\mu[w_1\cdots w_n1]\\
&=\mu[w_1\cdots w_n]\cdot\frac{\mu[w_{n-m+2}\cdots w_n0]}{\mu[w_{n-m+2}\cdots w_n]}+\mu[w_1\cdots w_n]\cdot\frac{\mu[w_{n-m+2}\cdots w_n1]}{\mu[w_{n-m+2}\cdots w_n]}\\
&\overset{(\star)}{=}\mu[w_1\cdots w_n],
\end{aligned}$$
where ($\star$) can be proved as follows.
\newline \textcircled{\footnotesize{a}} If $\mu[w_1\cdots w_n]=0$, then ($\star$) is obvious.
\newline \textcircled{\footnotesize{b}} If $\mu[w_{n-m+2}\cdots w_n]=0$, since the fact that $\mu$ is $(\beta,m,a)$-coordinated implies $\mu[w_{n-m+1}\cdots w_n]\le\mu[w_{n-m+2}\cdots w_n]$, we get $\mu[w_{n-m+1}\cdots w_n]=0$. Then
$$\mu[w_1\cdots w_n]=\mu[w_1\cdots w_{n-1}]\cdot\frac{\mu[w_{n-m+1}\cdots w_n]}{\mu[w_{n-m+1}\cdots w_{n-1}]}=0$$
and ($\star$) follows.
\newline \textcircled{\footnotesize{c}} If $\mu[w_1\cdots w_n]\neq0$ and $\mu[w_{n-m+2}\cdots w_n]\neq0$, then ($\star$) follows from
$$\mu[w_{n-m+2}\cdots w_n0]+\mu[w_{n-m+2}\cdots w_n1]=\mu[w_{n-m+2}\cdots w_n],$$
noting that $\mu$ is $(\beta,m,a)$-coordinated.
\item[ii)] If $w_1\cdots w_n1\notin\Sigma_\beta^*$, by Proposition \ref{every m} and $w_1\cdots w_n\in\Sigma_\beta^*$ we get $w_{n-m+2}\cdots w_n1\notin\Sigma_\beta^*$. Since $\mu$ is $(\beta,m,a)$-coordinated, we get $\mu[w_{n-m+2}\cdots w_n0]=\mu[w_{n-m+2}\cdots w_n]$ and then
$$\begin{aligned}
&\sum_{\substack{v=0,1\\w_1\cdots w_nv\in\Sigma_\beta^*}}\mu[w_1\cdots w_nv]=\mu[w_1\cdots w_n0]\\
&=\mu[w_1\cdots w_n]\cdot\frac{\mu[w_{n-m+2}\cdots w_n0]}{\mu[w_{n-m+2}\cdots w_n]}\overset{(\star)}{=}\mu[w_1\cdots w_n],
\end{aligned}$$
where ($\star$) follows in the same way as i) \textcircled{\footnotesize{b}} if $\mu[w_{n-m+2}\cdots w_n]=0$.
\end{itemize}
\textcircled{\footnotesize{2}} Prove $\sum_{\substack{u=0,1\\uw_1\cdots w_n\in\Sigma_\beta^*}}\mu[uw_1\cdots w_n]=\mu[w_1\cdots w_n]$ for all $w_1\cdots w_n\in\Sigma_\beta^n$ and $n\ge m$ by induction. Since $\mu$ is $(\beta,m,a)$-coordinated, the conclusion is true for $n=m-1$. Now suppose that the conclusion is already true for some $n\ge m-1$. We consider $n+1$ in the following. Let $w_1\cdots w_{n+1}\in\Sigma_\beta^{n+1}$.
\begin{itemize}
\item[i)] If $1w_1\cdots w_{n+1}\in\Sigma_\beta^*$, then $1w_1\cdots w_n\in\Sigma_\beta^*$ and
$$\begin{aligned}
&\sum_{\substack{u=0,1\\uw_1\cdots w_{n+1}\in\Sigma_\beta^*}}\mu[uw_1\cdots w_{n+1}]=\mu[0w_1\cdots w_{n+1}]+\mu[1w_1\cdots w_{n+1}]\\
&=\mu[0w_1\cdots w_n]\cdot\frac{\mu[w_{n-m+2}\cdots w_{n+1}]}{\mu[w_{n-m+2}\cdots w_n]}+\mu[1w_1\cdots w_n]\cdot\frac{\mu[w_{n-m+2}\cdots w_{n+1}]}{\mu[w_{n-m+2}\cdots w_n]}\\
&\overset{(\star)}{=}\mu[w_1\cdots w_n]\cdot\frac{\mu[w_{n-m+2}\cdots w_{n+1}]}{\mu[w_{n-m+2}\cdots w_n]}=\mu[w_1\cdots w_{n+1}]
\end{aligned}$$
where ($\star$) follows from inductive hypothesis.
\item[ii)] If $1w_1\cdots w_{n+1}\notin\Sigma_\beta^*$, by Proposition \ref{every m} and $w_1\cdots w_{n+1}\in\Sigma_\beta^*$ we get $1w_1\cdots w_n\notin\Sigma_\beta^*$, and then
$$\begin{aligned}
\sum_{\substack{u=0,1\\uw_1\cdots w_{n+1}\in\Sigma_\beta^*}}\mu[uw_1\cdots w_{n+1}]&=\mu[0w_1\cdots w_{n+1}]=\mu[0w_1\cdots w_n]\cdot\frac{\mu[w_{n-m+2}\cdots w_{n+1}]}{\mu[w_{n-m+2}\cdots w_n]}\\
&\overset{(\star)}{=}\mu[w_1\cdots w_n]\cdot\frac{\mu[w_{n-m+2}\cdots w_{n+1}]}{\mu[w_{n-m+2}\cdots w_n]}=\mu[w_1\cdots w_{n+1}]
\end{aligned}$$
where ($\star$) follows from inductive hypothesis.
\end{itemize}
(3) By the definition of $(\beta,m,a)$-coordinated set functions and $\frak{h}_\mu(\beta,m)$, it is straightforward to see that the supremum of
$$\Big\{\frak{h}_\mu(\beta,m):\mu\text{ is an }(\beta,m,a)\text{-coordinated set function}\Big\}$$
can be achieved as a maximum.
\end{proof}

\section{Proof of Theorem \ref{pseudo-golden ratios}}

We need the following lemma which follows immediately from the convexity of the function $x\log x$.

\begin{lemma}\label{inequality}
Let $\phi:[0,\infty)\to\R$ be defined by
$$\phi(x)=\left\{\begin{array}{ll}
0 & \mbox{if } x=0;\\
-x\log x & \mbox{if } x>0.
\end{array}\right.$$
Then for all $x,y\in[0,\infty)$ and $a,b\ge0$ with $a+b=1$,
$$a\phi(x)+b\phi(y)\le\phi(ax+by).$$
The equality holds if and only if $x=y$, $a=0$ or $b=0$.
\end{lemma}

\begin{proof}[Proof of Theorem \ref{pseudo-golden ratios}]
\item(1) By $\epsilon^*(1,\beta)=(1^{m-1}0)^\infty$ and Lemma \ref{charADM}, we know that for any $x\in[0,1)$, every $m$ consecutive digits in $\epsilon(x,\beta)$ must contain at least one $0$. This implies
    $$\#\{1\le k\le n:\varepsilon_k(x,\beta)=0\}\ge\lfloor\frac{n}{m}\rfloor$$
    for all $n\in\N$, and then
    $$\varliminf_{n\to\infty}\frac{\#\{1\le k\le n: \varepsilon_k(x,\beta)=0\}}{n}\ge\frac{1}{m}$$
    for any $x\in[0,1)$. If $0\le a<\frac{1}{m}$, we get $F_{\beta,a}=\emptyset$.
\item(2) When $\frac{1}{m}\le a\le1$, $f_a$ is a continuous function on its domain of definition
$$\begin{aligned}
D_{m,a}:=\Big\{(x_1,x_2,\cdots,x_{m-2}&)\in\R^{m-2}:\text{all terms in the }\log\text{'s in }f_a\text{ are non-negative}\Big\}\\
=\Big\{(x_1,x_2,\cdots,x_{m-2}&)\in\R^{m-2}:a\ge x_1\ge x_2\ge\cdots\ge x_{m-2}\ge0\quad\text{and}\\
&x_1+\cdots+x_{m-3}+x_{m-2}\le 1-a \le x_1+\cdots+x_{m-3}+2x_{m-2}\Big\},
\end{aligned}$$
which is closed and non-empty since
$$\left\{\begin{array}{ll}
(a,\frac{1-2a}{m-2},\cdots,\frac{1-2a}{m-2})\in D_{m,a} & \text{if } \frac{1}{m}\le a<\frac{1}{2};\\
(1-a,0,\cdots,0)\in D_{m,a} & \text{if } a\ge\frac{1}{2}.
\end{array}\right.$$
Therefore $\max_{(x_1,\cdots,x_{m-2})\in D_{m,a}} f_a(x_1,\cdots,x_{m-2})$ exists.

In order to get our conclusion, by Theorem \ref{Markov variational formula}, it suffices to prove
\begin{small}\begin{equation}\label{to prove 1.3.1}
\max\Big\{\frak{h}_\mu(\beta,m): \mu\text{ is an $(\beta,m,a)$-coordinated set function}\Big\}=\max_{(x_1,\cdots,x_{m-2})\in D_{m,a}} f_a(x_1,\cdots,x_{m-2})
\end{equation}\end{small}
in the following \textcircled{\footnotesize{$1$}} and \textcircled{\footnotesize{$2$}}, which are enlightened by drawing figures of the cylinders in $[0,1)$ and understanding their relations.
\newline\textcircled{\footnotesize{$1$}} Prove the inequality ``$\le$'' in (\ref{to prove 1.3.1}).
\newline Let $\mu$ be an $(\beta,m,a)$-coordinated set function. By Lemma \ref{charADM} we get $\Sigma_\beta^m=\{0,1\}^m\setminus\{1^m\}$, $\mu[1^{m-1}0]=\mu[1^{m-1}]$ and then
\begin{eqnarray*}
\frak{h}_\mu(\beta,m)&=&-\sum_{\substack{i_1,\cdots,i_m\in\{0,1\} \\ i_2\cdots i_{m-1}\neq1^{m-2}}}\mu[i_1\cdots i_m]\log\frac{\mu[i_1\cdots i_m]}{\mu[i_1\cdots i_{m-1}]}\\
& &-\mu[01^{m-2}0]\log\frac{\mu[01^{m-2}0]}{\mu[01^{m-2}]}-\mu[01^{m-1}]\log\frac{\mu[01^{m-1}]}{\mu[01^{m-2}]}.
\end{eqnarray*}
For $i_2\cdots i_{m-1}\neq1^{m-2}$ and $i_m\in\{0,1\}$, we can prove
\begin{small}\begin{equation}\label{can prove 2}
-\mu[0i_2\cdots i_m]\log\frac{\mu[0i_2\cdots i_m]}{\mu[0i_2\cdots i_{m-1}]}-\mu[1i_2\cdots i_m]\log\frac{\mu[1i_2\cdots i_m]}{\mu[1i_2\cdots i_{m-1}]}\\
\le-\mu[i_2\cdots i_m]\log\frac{\mu[i_2\cdots i_m]}{\mu[i_2\cdots i_{m-1}]}.
\end{equation}\end{small}
In fact, if $\mu[0i_2\cdots i_{m-1}]=0$, then $\mu[0i_2\cdots i_m]=0$. We get $\mu[1i_2\cdots i_{m-1}]=\mu[i_2\cdots i_{m-1}]-\mu[0i_2\cdots i_{m-1}]=\mu[i_2\cdots i_{m-1}]$ and $\mu[1i_2\cdots i_{m}]=\mu[i_2\cdots i_{m}]-\mu[0i_2\cdots i_{m}]=\mu[i_2\cdots i_{m}]$, which imply (\ref{can prove 2}). If $\mu[1i_2\cdots i_{m-1}]=0$, in the same way we can get (\ref{can prove 2}). If $\mu[0i_2\cdots i_{m-1}]\neq0$ and $\mu[1i_2\cdots i_{m-1}]\neq0$, then $\mu[i_2\cdots i_{m-1}]\neq0$ and (\ref{can prove 2}) follows from
\begin{eqnarray*}
& &-\mu[0i_2\cdots i_m]\log\frac{\mu[0i_2\cdots i_m]}{\mu[0i_2\cdots i_{m-1}]}-\mu[1i_2\cdots i_m]\log\frac{\mu[1i_2\cdots i_m]}{\mu[1i_2\cdots i_{m-1}]} \\
&=&\mu[i_2\cdots i_{m-1}]\Big(\frac{\mu[0i_2\cdots i_{m-1}]}{\mu[i_2\cdots i_{m-1}]}(-\frac{\mu[0i_2\cdots i_m]}{\mu[0i_2\cdots i_{m-1}]}\log\frac{\mu[0i_2\cdots i_m]}{\mu[0i_2\cdots i_{m-1}]})\\
& &\quad\quad\quad\quad\quad\quad+\frac{\mu[1i_2\cdots i_{m-1}]}{\mu[i_2\cdots i_{m-1}]}(-\frac{\mu[1i_2\cdots i_m]}{\mu[1i_2\cdots i_{m-1}]}\log\frac{\mu[1i_2\cdots i_m]}{\mu[1i_2\cdots i_{m-1}]})\Big) \\
&\le&-\mu[i_2\cdots i_m]\log\frac{\mu[i_2\cdots i_m]}{\mu[i_2\cdots i_{m-1}]},
\end{eqnarray*}
where the last inequality follows from Lemma \ref{inequality}. Thus
\begin{eqnarray*}
\frak{h}_\mu(\beta,m)&\le&-\sum_{\substack{i_2,\cdots,i_m\in\{0,1\}\\i_2\cdots i_{m-1}\neq1^{m-2}}}\mu[i_2\cdots i_m]\log\frac{\mu[i_2\cdots i_m]}{\mu[i_2\cdots i_{m-1}]} \\
& &-\mu[01^{m-2}0]\log\frac{\mu[01^{m-2}0]}{\mu[01^{m-2}]}-\mu[01^{m-1}]\log\frac{\mu[01^{m-1}]}{\mu[01^{m-2}]} \\
&=&-\sum_{\substack{i_1,\cdots,i_{m-1}\in\{0,1\}\\i_1\cdots i_{m-2}\neq1^{m-2}}}\mu[i_1\cdots i_{m-1}]\log\frac{\mu[i_1\cdots i_{m-1}]}{\mu[i_1\cdots i_{m-2}]}\\
& &-\mu[01^{m-2}0]\log\frac{\mu[01^{m-2}0]}{\mu[01^{m-2}]}-\mu[01^{m-1}]\log\frac{\mu[01^{m-1}]}{\mu[01^{m-2}]} \\
&=&-\sum_{\substack{i_1,\cdots,i_{m-1}\in\{0,1\}\\i_2\cdots i_{m-2}\neq1^{m-3}}}\mu[i_1\cdots i_{m-1}]\log\frac{\mu[i_1\cdots i_{m-1}]}{\mu[i_1\cdots i_{m-2}]}\\
& &-\mu[01^{m-3}0]\log\frac{\mu[01^{m-3}0]}{\mu[01^{m-3}]}-\mu[01^{m-2}]\log\frac{\mu[01^{m-2}]}{\mu[01^{m-3}]} \\
& &-\mu[01^{m-2}0]\log\frac{\mu[01^{m-2}0]}{\mu[01^{m-2}]}-\mu[01^{m-1}]\log\frac{\mu[01^{m-1}]}{\mu[01^{m-2}]}.
\end{eqnarray*}
For $i_2\cdots i_{m-2}\neq1^{m-3}$ and $i_{m-1}\in\{0,1\}$, in the same way as proving (\ref{can prove 2}), we get
\begin{align*}
-\mu[0i_2\cdots i_{m-1}]\log\frac{\mu[0i_2\cdots i_{m-1}]}{\mu[0i_2\cdots i_{m-2}]}-\mu[1i_2\cdots i_{m-1}]\log\frac{\mu[1i_2\cdots i_{m-1}]}{\mu[1i_2\cdots i_{m-2}]}\\
\le-\mu[i_2\cdots i_{m-1}]\log\frac{\mu[i_2\cdots i_{m-1}]}{\mu[i_2\cdots i_{m-2}]}.
\end{align*}
Thus
\begin{eqnarray*}
\frak{h}_\mu(\beta,m)&\le&-\sum_{\substack{i_2,\cdots,i_{m-1}\in\{0,1\}\\i_2\cdots i_{m-2}\neq1^{m-3}}}\mu[i_2\cdots i_{m-1}]\log\frac{\mu[i_2\cdots i_{m-1}]}{\mu[i_2\cdots i_{m-2}]}\\
& &-\mu[01^{m-3}0]\log\frac{\mu[01^{m-3}0]}{\mu[01^{m-3}]}-\mu[01^{m-2}]\log\frac{\mu[01^{m-2}]}{\mu[01^{m-3}]} \\
& &-\mu[01^{m-2}0]\log\frac{\mu[01^{m-2}0]}{\mu[01^{m-2}]}-\mu[01^{m-1}]\log\frac{\mu[01^{m-1}]}{\mu[01^{m-2}]} \\
&=&-\sum_{\substack{i_1,\cdots,i_{m-2}\in\{0,1\}\\i_1\cdots i_{m-3}\neq1^{m-3}}}\mu[i_1\cdots i_{m-2}]\log\frac{\mu[i_1\cdots i_{m-2}]}{\mu[i_1\cdots i_{m-3}]}\\
& &-\mu[01^{m-3}0]\log\frac{\mu[01^{m-3}0]}{\mu[01^{m-3}]}-\mu[01^{m-2}]\log\frac{\mu[01^{m-2}]}{\mu[01^{m-3}]} \\
& &-\mu[01^{m-2}0]\log\frac{\mu[01^{m-2}0]}{\mu[01^{m-2}]}-\mu[01^{m-1}]\log\frac{\mu[01^{m-1}]}{\mu[01^{m-2}]} \\
&=&-\sum_{\substack{i_1,\cdots,i_{m-2}\in\{0,1\}\\i_2\cdots i_{m-3}\neq1^{m-4}}}\mu[i_1\cdots i_{m-2}]\log\frac{\mu[i_1\cdots i_{m-2}]}{\mu[i_1\cdots i_{m-3}]}\\
& &-\mu[01^{m-4}0]\log\frac{\mu[01^{m-4}0]}{\mu[01^{m-4}]}-\mu[01^{m-3}]\log\frac{\mu[01^{m-3}]}{\mu[01^{m-4}]} \\
& &-\mu[01^{m-3}0]\log\frac{\mu[01^{m-3}0]}{\mu[01^{m-3}]}-\mu[01^{m-2}]\log\frac{\mu[01^{m-2}]}{\mu[01^{m-3}]} \\
& &-\mu[01^{m-2}0]\log\frac{\mu[01^{m-2}0]}{\mu[01^{m-2}]}-\mu[01^{m-1}]\log\frac{\mu[01^{m-1}]}{\mu[01^{m-2}]}.
\end{eqnarray*}
\newline$\cdots$
\newline Repeat the above process a finite number of times. Finally we get
\begin{eqnarray*}
\frak{h}_\mu(\beta,m)&\le&-\mu[00]\log\frac{\mu[00]}{\mu[0]}-\mu[01]\log\frac{\mu[01]}{\mu[0]} \\
& &-\mu[010]\log\frac{\mu[010]}{\mu[01]}-\mu[011]\log\frac{\mu[011]}{\mu[01]} \\
& &\cdots \\
& &-\mu[01^{m-3}0]\log\frac{\mu[01^{m-3}0]}{\mu[01^{m-3}]}-\mu[01^{m-2}]\log\frac{\mu[01^{m-2}]}{\mu[01^{m-3}]} \\
& &-\mu[01^{m-2}0]\log\frac{\mu[01^{m-2}0]}{\mu[01^{m-2}]}-\mu[01^{m-1}]\log\frac{\mu[01^{m-1}]}{\mu[01^{m-2}]}.
\end{eqnarray*}
Since $\mu$ is $(\beta,m,a)$-coordinated, we have
$$\left\{\begin{array}{ll}
\mu[0]=a, & \mu[1]=1-a, \\
\mu[00]+\mu[01]=\mu[0], & \mu[01]+\mu[11]=\mu[1], \\
\mu[010]+\mu[011]=\mu[01], & \mu[011]+\mu[111]=\mu[11], \\
\cdots, & \cdots, \\
\mu[01^{m-3}0]+\mu[01^{m-2}]=\mu[01^{m-3}], & \mu[01^{m-2}]+\mu[1^{m-1}]=\mu[1^{m-2}], \\
\mu[01^{m-2}0]+\mu[01^{m-1}]=\mu[01^{m-2}], & \mu[01^{m-1}]=\mu[1^{m-1}].
\end{array}\right.$$
Let $y_1:=\mu[01], y_2:=\mu[011], \cdots, y_{m-2}:=\mu[01^{m-2}]$. Then we have
\begin{small}
$$\left\{\begin{array}{l}
\mu[0]=a, \mu[00]=a-y_1, \mu[010]=y_1-y_2, \mu[0110]=y_2-y_3, \cdots, \mu[01^{m-3}0]=y_{m-3}-y_{m-2}, \\
\mu[1]=1-a, \mu[11]=1-a-y_1, \cdots, \mu[1^{m-1}]=1-a-y_1-y_2-\cdots-y_{m-2}, \\
\mu[01^{m-1}]=1-a-y_1-y_2-\cdots-y_{m-2}, \mu[01^{m-2}0]=y_1+y_2+\cdots+y_{m-3}+2y_{m-2}+a-1.
\end{array}\right.$$
\end{small}
By a simple calculation, we get
$$\frak{h}_\mu(\beta,m)\le f_a(y_1,\cdots,y_{m-2}).$$
It follows from $\mu[00],\mu[010],\cdots,\mu[01^{m-3}0],\mu[01^{m-2}0],\mu[01^{m-1}]\ge0$ that $(y_1,\cdots,y_{m-2})\in D_{m,a}$. Therefore
$$\frak{h}_\mu(\beta,m)\le\max_{(x_1,\cdots,x_{m-2})\in D_{m,a}} f_a(x_1,\cdots,x_{m-2}).$$
\newline\textcircled{\footnotesize{$2$}} Prove that the inequality ``$\le$'' in (\ref{to prove 1.3.1}) can achieve ``$=$'' by some $(\beta,m,a)$-coordinated set function.
\newline Let $(y_1,\cdots,y_{m-2})\in D_{m,a}$ such that $$f_a(y_1,\cdots,y_{m-2})=\max_{(x_1,\cdots,x_{m-2})\in D_{m,a}} f_a(x_1,\cdots,x_{m-2}).$$
Define
\begin{small}
\begin{eqnarray*}
\mu[0]:=a,& &\mu[1]:=1-a,\\
\mu[00]:=a-y_1,&\mu[01]=\mu[10]:=y_1,&\mu[11]:=1-a-y_1,\\
\mu[010]:=y_1-y_2,&\mu[011]=\mu[110]:=y_2,&\mu[111]:=1-a-y_1-y_2,\\
\cdots,&\cdots,&\cdots,\\
\mu[01^{m-3}0]:=y_{m-3}-y_{m-2},&\mu[01^{m-2}]=\mu[1^{m-2}0]:=y_{m-2},&\mu[1^{m-1}]:=1-a-y_1-\cdots-y_{m-2},
\end{eqnarray*}
$$\mu[01^{m-2}0]:=y_1+\cdots+y_{m-3}+2y_{m-2}+a-1,\mu[01^{m-1}]=\mu[1^{m-1}0]:=1-a-y_1-\cdots-y_{m-2}$$
\end{small}
and
\begin{eqnarray}\label{uwv}
\mu[uwv]:=\frac{\mu[uw]\cdot\mu[wv]}{\mu[w]}\quad\text{for }u,v\in\{0,1\}\quad\text{and}\quad w\in\bigcup_{k=1}^{m-2}\Big(\{0,1\}^k\setminus\{1^k\}\Big)
\end{eqnarray}
where $\mu[uwv]$ is defined to be $0$ if one of $\mu[w]$, $\mu[uw]$ and $\mu[wv]$ is $0$. Then $\mu$ is an $(\beta,m,a)$-coordinated set function. By (\ref{uwv}) and Lemma \ref{inequality}, it is straightforward to check that in the proof of \textcircled{\footnotesize{$1$}}, all the ``$\le$'' in the upper bound estimation of $\frak{h}_\mu(\beta,m)$ can take ``$=$'' and then
$$\frak{h}_\mu(\beta,m)=f_a(y_1,\cdots,y_{m-2})=\max_{(x_1,\cdots,x_{m-2})\in D_{m,a}} f_a(x_1,\cdots,x_{m-2}).$$
\end{proof}

\begin{ack}
The author is grateful to Professor Jean-Paul Allouche and Professor Bing Li for their advices on a former version of this paper, and also grateful to the Oversea Study Program of Guangzhou Elite Project (GEP) for financial support (JY201815).
\end{ack}


\begin{thebibliography}{00}

\bibitem{ABK19} R. Alcaraz Barrera, S. Baker, and D. Kong, \textit{Entropy, topological transitivity, and dimensional properties of unique $q$-expansions}, Trans. Amer. Math. Soc. 371 (2019) 3209--3258.

\bibitem{B18} S. Baker, \textit{Digit frequencies and self-affine sets with non-empty interior}, Ergodic Theory Dynam. Systems, (First published online 2018), 1--33.

\bibitem{BK19} S. Baker and D. Kong, \textit{Numbers with simply normal $\beta$-expansions}, Math. Proc. Cambridge Philos. Soc. 167 (2019), no. 1, 171--192.

\bibitem{B89} F. Blanchard, \textit{$\beta$-expansions and symbolic dynamics}, Theoret. Comput. Sci. 65 (1989) 131--141.

\bibitem{B09} E. Borel, \textit{Les probabilit\'es d\'enombrables et leurs applications arithm\'etiques}, Rend. Circ. Mat. Palermo (2) 27 (1909) 247--271.

\bibitem{B73} R. Bowen, \textit{Topological entropy for noncompact sets}, Trans. Amer. Math. Soc. 184 (1973) 125--136.

\bibitem{BW14} Y. Bugeaud and B.-W. Wang, \textit{Distribution of full cylinders and the Diophantine properties of the orbits in $\beta$-expansions}, J. Fractal Geom. 1 (2014) 221--241.

\bibitem{DK09} M. de Vries and V. Komornik, \textit{Unique expansions of real numbers}, Adv. Math. 221 (2009) 390--427.

\bibitem{E49} H. Eggleston, \textit{The fractional dimension of a set defined by decimal properties}, Q. J. Math. 20 (1949) 31--36.

\bibitem{F90} K. J. Falconer, \textit{Fractal geometry}, John Wiley $\&$ Sons, Ltd., Chichester, 1990. Mathematical foundations and applications.

\bibitem{FFW01} A.-H. Fan, D.-J. Feng, and J. Wu, \textit{Recurrence, dimension and entropy}, J. Lond. Math. Soc. (2) 64 (2001) 229--244.

\bibitem{FW12} A.-H. Fan and B.-W. Wang, \textit{On the lengths of basic intervals in beta expansions}, Nonlinearity 25 (2012) 1329--1343.

\bibitem{FZ04} A. Fan and H. Zhu, \textit{Level sets of $\beta$-expansions}, C. R. Math. Acad. Sci. Paris 339 (2004) 709--712.

\bibitem{FWL19} L. Fang, M. Wu, and B. Li, \textit{Approximation orders of real numbers by $\beta$-expansions}, Math. Z. (2019) 1--28.

\bibitem{FS92} C. Frougny and B. Solomyak, \textit{Finite beta-expansions}, Ergodic Theory Dynam. Systems 12 (1992) 713--723.

\bibitem{K97} B. P. Kitchens, \textit{Symbolic dynamics}, Universitext, Springer-Verlag, Berlin, 1998. One-sided, two-sided and countable state Markov shifts.

\bibitem{LLS19} B. Li, Y.-Q. Li, and T. Sahlsten, \textit{Random walks associated to beta-shifts}, arXiv:1910.13006

\bibitem{LPWW14} B. Li, T. Persson, B. Wang, and J. Wu, \textit{Diophantine approximation of the orbit of $1$ in the dynamical system of beta expansions}, Math. Z. 276 (2014) 799--827.

\bibitem{LL16} J. J. Li and B. Li, \textit{Hausdorff dimension of some irregular sets associated with $\beta$-expansions}, Sci. China Math. 59 (2016) 445--458.

\bibitem{L20} Y.-Q. Li, \textit{Digit frequencies of beta-expansions}, Acta Math. Hungar. 162 (2020), no. 2, 403--418.

\bibitem{LL18} Y.-Q. Li and B. Li, \textit{Distributions of full and non-full words in beta-expansions}, J. Number Theory 190 (2018) 311--332.

\bibitem{P60} W. Parry, \textit{On the $\beta$-expansions of real numbers}, Acta Math. Hungar. 11 (1960) 401--416.

\bibitem{PS05} C.-E. Pfister and W.G. Sullivan, \textit{Large deviations estimates for dynamical systems without the specification property. Applications to the $\beta$-shifts}, Nonlinearity 18 (2005) 237--261.

\bibitem{R57} A. R\'enyi, \textit{Representations for real numbers and their ergodic properties}, Acta Math. Hungar. 8 (1957) 477--493.

\bibitem{S97} J. Schmeling, \textit{Symbolic dynamics for $\beta$-shifts and self-normal numbers}, Ergodic Theory Dynam. Systems 17 (1997) 675--694.

\bibitem{S80} K. Schmidt, \textit{On periodic expansions of Pisot numbers and Salem numbers}, Bull. Lond. Math. Soc. 12 (1980) 269--278.

\bibitem{S03} N. Sidorov, \textit{Almost every number has a continuum of $\beta$-expansions}, Amer. Math. Monthly 110 (2003) 838--842.

\bibitem{TV03} F. Takens and E. Verbitskiy, \textit{On the variational principle for the topological entropy of certain non-compact sets}, Ergodic Theory Dynam. Systems 23 (2003) 317--348.

\bibitem{T01} A. A. Tempelman, \textit{Multifractal analysis of ergodic averages: a generalization of Eggleston's theorem}, J. Dyn. Control Syst. 7 (2001) 535--551.

\bibitem{T12} D. Thompson, \textit{Irregular sets, the $\beta$-transformation and the almost specifcation property}, Trans. Amer. Math. Soc. 364 (2012) 5395--5414.

\bibitem{W82} P. Walters, \textit{An Introduction to Ergodic Theory}, Graduate Texts in Mathematics, vol. 79, Springer-Verlag, New York-Berlin, 1982.

\end{thebibliography}
\end{document}